\documentclass[12pt]{amsart}
\usepackage{amsfonts}
\usepackage{tikz-cd}
\usepackage{amscd}
\usepackage{amsmath, amsfonts,amssymb,amsthm,mathrsfs,xypic}
\usepackage{graphicx}
\xyoption{curve}
\usepackage{fancybox}
\numberwithin{equation}{section}

\theoremstyle{plain}
\newtheorem{theorem}[equation]{Theorem}
\newtheorem{proposition}[equation]{Proposition}
\newtheorem{lemma}[equation]{Lemma}
\newtheorem{corollary}[equation]{Corollary}

\theoremstyle{definition}
\newtheorem{definition}[equation]{Definition}

\theoremstyle{remark}
\newtheorem{remark}[equation]{Remark}

\newcommand{\To}{\longrightarrow}

\def\C{\mathcal{C}}
\def\H{\mathcal{H}}

\def\D{\mathcal{D}}

\def\W{\mathcal{W}}

\def\F{\mathcal{F}}

\begin{document}
\title[\tiny{A note on model (co-)slice categories}]{A note on model (co-)slice categories}
\author [\tiny{Zhi-Wei Li}] {Zhi-Wei Li}
\date{\today}

\thanks{The author is supported by JSNU12XLR025.}

\date{\today}
%\thanks{$^*$ Corresponding to: }

\maketitle

\begin{abstract} \ There are various adjunctions between model (co-)slice categories. We characterize when these adjunctions are Quillen equivalences. As an application, a triangle equivalence between the stable category of a Frobenius category and the homotopy category of a non-pointed model category is given.
\end{abstract}

\setcounter{tocdepth}{1}

\section{Introduction}
 Given a category $\C$ and a morphism $f:X\to Y$ in $\C$, one can construct adjunctions between coslice categories $(X\downarrow \C)$ and $(Y\downarrow \C)$, and between slice categories $(\C\downarrow X)$ and $(\C\downarrow Y)$. If $\C$ is a closed model category, all these (co-)slice categories inherit a model structure from $\C$. In this case, the adjunctions between these (co-)slice categories are Quillen adjunctions. Meanwhile, if we start from an Quillen adjunction between two closed model categories $\C$ and $\D$, we also can construct Quillen adjunctions between some (co-)slice categories. This note is aimed to characterize when these adjunctions are Quillen equivalences.

 As an application of the characterizations, we construct a triangle equivalence between the stable category of a weakly idempotent complete Frobenius model category and the homotopy category of its coslice category; see Corollary 3.6. As a byproduct, we get a non-pointed model category whose homotopy category is a triangulated category; see Theorem 3.5. This shows that the pointed condition of Quilen's Theorem \cite[Chapter I, Theorem 2]{Quillen67} is not always necessary.

\section{Preliminaries of (co)slice categories and Quillen equivalences}
In this section we recall some basic facts about (co)slice categories and Quillen equivalences. Our main references are \cite[Chapter I]{Quillen67}, \cite{Dwyer/Spalinski95}, \cite[Chapter 1 ]{Hovey99} and \cite[Chapter 8]{Hirschhorn03}.

\subsection{Quillen equivalences}

In a closed model category $\C$ \cite[Definition I.5.1]{Quillen67}, there are three classes of morphisms, called {\it cofibrations, fibrations} and {\it weak equivalences}. We will denote them by $\C of(\C), \F ib(\C)$ and $\W e(\C)$, respectively. A morphism which is both a (co-)fibration and a weak equivalence is called {\it acyclic (co-)fibration}. An object $X\in \C$ is called {\it cofibrant} if $\emptyset\to A\in \C of(\C)$ and {\it fibrant} if $X\to *\in \F ib(\C)$, where $\emptyset$ is the initial object of $\C$ and $*$ the terminal object of $\C$. We use $\C_c$ and $\C_f$ to denote the classes of cofibrant and fibrant objects, respectively. An object in $\C_{cf}:=\C_c\cap \C_f$ is called {\it bifibrant}.

 Suppose $\C$ and $\D$ are closed model categories. An adjunction $F: \C\leftrightarrows \D : G$ is called a {\it Quillen adjunction} if $F$ preserves cofibrations and acyclic cofibrations or equivalently $G$ preserves fibrations and acyclic fibrations; see \cite[Definition 1.3.1, Lemma 1.3.4]{Hovey99}. Sometimes we will call $F$ {\it a left Quillen functor } and $G$ {\it a right Quillen functor}.
\vskip5pt
\begin{definition} \ \cite[Definition 1.3.12]{Hovey99} \ A Quillen adjunction $(F,G;\varphi): \C \to\D$ is called a {\it Quillen equivalence} if for all $X\in \C_c$
 and $Y\in \D_f$, a map $f: F(X)\to Y\in \mathcal{W}e(\D)$ if and only if $\varphi(f): X\to G(Y)\in \mathcal{W}e(\C)$.
 \end{definition}

If $(F,G)$ is a Quillen equivalence, then the left derived functor $\mathbb{L}F$ and the right derived functor $\mathbb{R}G$ exist; see \cite[Chapter I, Section 4]{Quillen67}. Furthermore, they induce an equivalent adjunction $(\mathbb{L}F, \mathbb{R}L):\H o(\C)\to \H o(\D)$ between the homotopy categories; see \cite[Chapter I, Theorem 3]{Quillen67}.

In a model category $\C$, we use $p_{_X}: Q(X)\to X$ to denote the cofibrant approximation of $X$ and  $r_{_X}:X\to R(X)$ the fibrant approximation of $X$, respectively; see \cite[Chapter I, Section 1]{Quillen67} or \cite[Section 5]{Dwyer/Spalinski95}.
The following is the most useful criterion for checking when a given Quillen adjunction is a Quillen equivalence.

\begin{proposition} \ \cite[Proposition 1.3.13, Corollary 1.3.16]{Hovey99} \ Suppose $(F,G,\varphi; \eta, \varepsilon):\C\to \D$ is a Quillen adjunction. The following are equivalent:
\vskip5pt
$(1)$ \ $(F, G, \varphi)$ is a Quillen equivalence.
\vskip5pt
$(2)$ \ $(\mathbb{L}F, \mathbb{R}G):\H o(\C)\to \H o(\D)$ is an adjoint equivalence of categories.
\vskip5pt

$(3)$ \ If $F(f)$ is a weak equivalence for a map $f$ in $\C_c$, so is $f$. And the map $F(Q(G(Y)))\stackrel{F(p_{_{G(Y)}})}\To FG(Y)\stackrel{\varepsilon_{_Y}}\To Y$ is a weak equivalence for every $Y\in \D_f$.

\vskip5pt
$(4)$ \ If $G(g)$ is a weak equivalence for a map $g$ in $\D_f$, so is $g$. And the map $X \stackrel{\eta_{_X}}\To GF(X)\stackrel{G(r_{_{F(X)}})}\To G(R(F(X)))$ is a weak equivalence for every $X\in \C_c$.
\end{proposition}

\subsection{The model (co)slice categories}

\begin{definition} \ Let $\C$ be a category. For an object $X$ in $\C$, the {\it coslice category} $(X\downarrow \C)$ is the category in which an object is a map $X\stackrel{f}\to C$ in $\C$, and a map from $X\stackrel{f}\to C$ to $X\stackrel{f'}\to C'$ is a map $\alpha:C\to C'$ such that $f'=\alpha\circ f$. The composition of maps is defined by the composition of maps in $\C$.
\end{definition}

Dually, we define the {\it slice category} $(\C\downarrow X)$ over $X$.
\vskip10pt

 Now let $\C$ be a closed model category. If we define a map in $(X\downarrow \C)$ and $(\C\downarrow X)$ is weak equivalence, cofibration, or fibration if it is one in $\C$, then both the coslice category $(X\downarrow \C)$ and the slice category $(\C\downarrow X)$ become closed model categories; see \cite[Theorem 7.6.5]{Hirschhorn03}.
\vskip10pt
\emph{From now on, when we talk about model coslice and slice categories, we always mean that their model structures are given as above.}
\vskip10pt
\begin{lemma} \ Let $\C$ be a model category. Then
\vskip5pt
$(1)$ \ $(X\downarrow \C)_c=\{u\in (X\downarrow \C)\ | \ u\in \C of(\C)\}$ and $(X\downarrow \C)_f=\{ X\stackrel{u}\to C \in (X\downarrow \C)\ | \ C\in \C_f\}$.
\vskip5pt
$(2)$ \ $(\C\downarrow X)_c=\{ C\stackrel{u}\to X \in (\C\downarrow X)\ |\ C\in \C_c\}$ and
$(\C \downarrow X)_f=\{u\in (\C\downarrow X) |\ u \in \mathcal{F}ib(\C)\}$.
\end{lemma}
\begin{proof} \ $(1)$.  Note that, the initial object in $(X\downarrow \C)$ is $X\stackrel{1_X}\to X$ and the terminal object is $X\to *$. From these the statement $(1)$ can be verified directly. The proof of assertion $(2)$ is dually.
\end{proof}

\subsection{Quillen adjunctions between (co-)slice categories}
Now let $\C$ be a bicomplete category and let $g: X\to Y$ be a map in $\C$. Then there are adjunctions $(g_!, g^*): (X\downarrow \C)\to (Y\downarrow \C)$ and $(g_*, g^!): (\C\downarrow X)\to (\C\downarrow Y)$; see \cite[Lemma 7.6.6]{Hirschhorn03}. Where $g_!$ takes the object $X\to C$ to its cobase change along $g$, and $g^*$ takes the object $Y\to D$ to its composition with $g$. Dually, $g_*$ takes the object $C\to X$ to its composition with $g$, and $g^!$ takes the object $D\to Y$ to its base change along $g$. In particular, if we take $X=\emptyset$ the initial object of $\C$, then $g^*$ is just the forgetful functor from $(Y\downarrow\C)$ to $\C=(\emptyset\downarrow \C)$.

If we already have an adjunction $(S, U; \varphi, \eta, \varepsilon): \C\to \D $ between the categories $\C$ and $\D$. Then for any object $X\in \C$ and $Y\in \D$, $(S, U)$ induces the following two adjunctions:
\vskip5pt
$\bullet$ \ $\overline{S}: (X\downarrow \C)\to (S(X)\downarrow \D)$ with $u\mapsto S(u)$ and $f:u\to u'\mapsto S(f)$; $\overline{U}: (S(X)\downarrow \D)\to (X\downarrow \C)$ with $S(X)\stackrel{u}\to D\mapsto X\stackrel{\eta_{_X}}\to US(X)\stackrel{U(u)}\to U(D)$ and $f:u\to u'\mapsto U(f)$.

\vskip5pt
$\bullet$ \ $\widetilde{S}: (\C \downarrow U(Y))\to (\D \downarrow Y)$ with $C\to U(Y)\mapsto S(C)\to SU(Y)\stackrel{\epsilon_{_Y}}\to Y$ and $f:u\to u'\mapsto S(f)$; $\widetilde{U}: (\D\downarrow Y)\to (\C \downarrow U(Y))$ with $u \mapsto U(u)$ and $f:u\to u'\mapsto U(f)$.

\vskip10pt

With above notations, we have the following proposition, part of which in some special case can be found in \cite[Chapter 1, Section 3]{Hovey99} and \cite[Chapter 16, Section 2]{May/Ponto12}.

\begin{proposition} \ $(1)$ \ Let $\C$ be a model category and $g: X\to Y$ a map in $\C$. Then
\vskip5pt
$(i)$ \  $(g_!, g^*): (X\downarrow \C)\to (Y\downarrow \C)$ is a Quillen adjunction.
\vskip5pt
$(ii)$ \ $(g_*, g^!): (\C\downarrow X)\to (\C\downarrow Y)$ is a Quillen adjunction.

\vskip5pt
$(2)$ \ Let $\C$ and $\D$ be two model categories. If $(S, U):\C\to \D$ is a Quillen adjunction and $X\in \C, Y\in\D$, then

\vskip5pt
$(i)$ \  $(\overline{S}, \overline{U}): (X\downarrow \C)\to (S(X)\downarrow \D)$ is a Quillen adjunction.
\vskip5pt
$(ii)$ \  $(\widetilde{S}, \widetilde{U}): (\C\downarrow U(Y))\to (\D\downarrow Y)$ is a Quillen adjunction.

\end{proposition}
\begin{proof} \ These two statements can be proved by the definition of Quillen adjunction and the pushout, pullback axioms of model categories. We leave the details to the reader. \end{proof}

Assume now that $\C$ and $\D$ are two closed model categories and denote by $\C_*=(*\downarrow \C)$ and $\D_*=(*\downarrow \D)$ the slice categories of $\C$ and $\D$ induced by the terminal object $*$, respectively. If $(S,U):\C\to \D$ is a Quillen adjunction, M. Hovey has constructed a functor $U_*$ from $\D_*\to \C_*$ by mapping object $*\stackrel{v}\to D$ to $U(*)=*\stackrel{U(v)}\to U(X)$ \cite{Hovey99}. He has shown that this functor is a right Quillen functor. Note that if we denote by the map $S(*)\to *$ as $g$ in $\D_*$, then $ U_*$ is the composition the functors $\D_*\stackrel{g^*}\to (S(*)\downarrow \D)\stackrel{\overline{U}}\to \C_*$. So by Proposition 2.5, this functor has a left adjoint $S_*=\overline{S}\circ g_!$, then we can get Proposition 1.3.5 of \cite{Hovey99} directly.

\begin{corollary} \cite[Proposition 1.3.5]{Hovey99}\ A Quillen adjunction $(S, U): \C\to \D$ induces a Quillen adjunction $(S_*, U_*): \C_*\to \D_*$.
\end{corollary}

\vskip10pt
\section{Main results }

 In this section we will characterize when the various Quillen adjunctions in Proposition 2.5 are Quillen equivalences.

\begin{proposition}  \ If $\C$ is a closed model category and $g:X\to Y$ a map in $\C$, then
\vskip5pt
$(i)$ \ $(g_!, g^*): (X\downarrow \C)\to (Y\downarrow \C)$ is a Quillen equivalence if and only if the cobase change of $g$ along $u$ for each $u\in (X\downarrow \C)_c$ is a weak equivalence.
\vskip5pt
$(ii)$ \ $(g_*, g^!): (\C \downarrow X)\to (\C \downarrow Y)$ is a Quillen equivalence if and only if the base change of $g$ along $u$ for each $u\in (\C\downarrow Y)_f$ is a weak equivalence.
\vskip5pt
$(2)$ \ Assume that $\C$ and $\D$ are closed model categories and $(S,U):\C\to \D$ a Quillen equivalence. Then given any object $X\in \C$ and $Y\in \D$,
\vskip5pt
$(i)$ \ if $X\in \C_c$, the adjunction $(\overline{S}, \overline{U}):(X\downarrow\C)\to (S(X)\downarrow \D)$ is a Quillen equivalence;
\vskip5pt
$(ii)$ \ if $Y\in \D_f$, the adjunction  $(\widetilde{S}, \widetilde{U}):(\C\downarrow U(Y))\to (\D\downarrow Y)$ is a Quillen equivalence.
\end{proposition}

\begin{proof} \ For the proof of statement $(i)$ of $(1)$, recall that $(X\downarrow \C)_c=\{ u\in (X\downarrow \C) \ | \ u\in \C of(\C)\}$ and $(Y\downarrow\C)_f=\{Y\stackrel{u}\to C\in (Y\downarrow \C) \ |\ C\in \C_f \}$. By the construction of $g_!$, the unit $\eta_{u}: u\to g^*g_!(u)$ is the cobase change of $g$ along $u$ for any $u\in (X\downarrow \C)$. Since $g^*(f)=f$, by Proposition 2.2, $(g_!, g^*)$ is a Quillen equivalence if and only if the composite $$u\stackrel{\eta_u}\To g^*g_!(u)\stackrel{g^*(r_{_{g_!(u)}})}\To g^*(R(g_!(u))) $$
is a weak equivalence for $u\in (X\downarrow \C)_c$. Note that $g^*(r_{_{g_!(u)}})=r_{_{g_!(u)}}$ is a weak equivalence. So by the 2-out-of-3 axiom of weak equivalences, this is equivalent to the cobase change of $g$ along $u$ is a weak equivalence.

The others can be proved similarly.
\end{proof}

\begin{corollary} \ \cite[Proposition 1.3.7]{Hovey99} \ Suppose $(S, U):\C\to \D$ is a Quillen equivalence, and suppose in addition that the terminal object $*\in \C_c$ and that $S$ preserves the terminal object. Then $(S_*, U_*): \C_*\to \D_*$ constructed as in Corollary $2.6$ is a Quillen equivalence.
\end{corollary}

\begin{proof} \ In this case, $(S_*, U_*)=(\overline{S}, \overline{U})$ is a Quillen equivalence by Proposition 3.1.
\end{proof}

Recall that  a closed model category $\C$ is called {\it left proper} if every cobase change of a weak equivalence along a cofibration is weak equivalence. Dually, $\C$ is called {\it right proper} if every base change of a weak equivalence along a fibration is a weak equivalence. By Proposition 3.1, we can redescribe the left or right properness of a model category by Quillen equivalences:

\begin{corollary}\footnote{This should be the right version of Proposition 16.2.4 of \cite{May/Ponto12}, there the authors claim that a closed model category $\C$ is left proper or right proper iff $(g_!,g^*)$ or $(g_*,g^!)$ is Quillen equivalence for a given map $g$. For a counter example, see Remark 3.4.} \ $(1)$ \  A closed model category $\C$ is left proper if and only if $(g_!,g^*)$ is a Quillen equivalence for every weak equivalence $g:X\to Y$.
\vskip5pt
$(2)$ \ A closed model category $\C$ is right proper if and only if $(g_*,g^!)$ is a Quillen equivalence for every weak equivalence $g:X\to Y$.

\end{corollary}

\begin{remark} \ In general, if $g$ is not a weak equivalence, even $\C$ is proper, $(g_!,g^*)$ is not necessarily a Quillen equivalence. For example, take $\C={\rm mod}{k[x]/(x^2)}$ where $k$ is a field. Then $\C$ is a proper model category in which weak equivalences are stable equivalences and every object is bifibrant. Take $g=0\to k$, then $\eta=\left(\begin{smallmatrix}
1 \\
0
\end{smallmatrix}\right)$ is the unit of the adjunction $(g_!,g^*)$. In this case every object is fibrant, the maps in $(4)$ of Proposition 2.2 is just $\eta_{_C}$ for any $C\in \C$. But $\eta_k:k\to k\oplus k$ is no way to be a weak equivalence. So the Quillen adjunction $(g_!,g^*): \C\to (k\downarrow \C)$ is not a Quillen equivalence.
\end{remark}

\vskip10pt

If $\F$ is a weakly idempotent complete Frobenius category, then $\F$ has a canonical model structure in which the cofibrations are the monomorphisms, the fibrations are the epimorphisms and the weak equivalences are the stable equivalences \cite[Theorem 3.3]{Gillespie2011}. Let $A$ be any nonzero projective-injective object in $\F$. Take $g=0:0\to A$, then $0_!(0\to C)=C\stackrel{\left(\begin{smallmatrix}
0 \\
1
\end{smallmatrix}\right)}\to A\oplus C$ is a weak equivalence for any $C\in \F$. By Proposition 3.1, we have a Quillen equivalence $(0_!, 0^*):(0\downarrow\F)=\F\to (A\downarrow \F).$  So the derived functors of $0_!$ and $0^*$ are equivalences of homotopy categories between $\H o(\F)$ and $\H o(A\downarrow \F)$. Note that in this case, the homotopy category $\H o(\F)$ is just the stable category $\underline{\F}$ \cite[Chapter I, Section 2.2]{Happel88} which is a triangulated category by Theorem 2.6 of \cite{Happel88}. If we can show that the homotopy category $\H o(A\downarrow \F)$ is a triangulated category, then the derived adjunction $(\mathbb{L}0_!, \mathbb{R}0^*):\underline{\F}\to \H o(A\downarrow \F)$ will be a triangle equivalence since Quillen equivalences are automatically triangle equivalences if the corresponding homotopy categories are triangulated categories; see \cite[Chapter I, Theorem 3]{Quillen67}.

Next we will show that the homotopy category $\H o(A\downarrow \F)$ is a triangulated category. And then we give the promised example as advertised in Introduction since the coslice category $(A\downarrow \F)$ is not pointed by noting that its initial object is $A\stackrel{1_A}\to A$ and its terminal object is $A\to 0$.

\begin{theorem} \  The homotopy category $\H o(A\downarrow \F)$ is a triangulated category.
\end{theorem}
\begin{proof} \  Since $\F_{cf}=\F$ and $A$ is injective, we know that $(A\downarrow \F)_f=(A\downarrow \F)$ and cofibrant objects in $(A\downarrow \F)$ are split monomorphisms in $\F$ with domain $A$. So for any object $u \in (A\downarrow \F)_c$, we may write $u$ as $A\stackrel{\left(\begin{smallmatrix}
1 \\
0
\end{smallmatrix}\right)}\to A\oplus C$ up to isomorphism. The morphisms from $u=A\stackrel{\left(\begin{smallmatrix}
1 \\
0
\end{smallmatrix}\right)}\to A\oplus C$ to $v=A\stackrel{\left(\begin{smallmatrix}
1 \\
0
\end{smallmatrix}\right)} \to A\oplus D$ are of the form $\left(\begin{smallmatrix}
1 & r\\
0 & s
\end{smallmatrix}\right)$. By Quillen's homotopy category theorem \cite[Theorem I.1]{Quillen67}, we know that the homotopy category $\H o(A\downarrow \F)$ can be realized as the quotient category $(A\downarrow \F)_c/\sim$ where $\sim$ is the homotopy relation. For details, we refer the reader to Section 1 of Chapter I of \cite{Quillen67} or Sections 4-5 of \cite{Dwyer/Spalinski95}.

Let $(A\downarrow \F)_c^0$ be the subcategory of $(A\downarrow \F)_c$ consisting of the objects of the form $A\stackrel{\left(\begin{smallmatrix}
1 \\
0
\end{smallmatrix}\right)}\to A\oplus C$ and morphisms from $u=A\stackrel{\left(\begin{smallmatrix}
1 \\
0
\end{smallmatrix}\right)}\to A\oplus C$ to $v=A\stackrel{\left(\begin{smallmatrix}
1 \\
0
\end{smallmatrix}\right)} \to A\oplus D$ are of the form $\left(\begin{smallmatrix}
1 & 0\\
0 & s
\end{smallmatrix}\right)$. The subcategory $(A\downarrow\F)_c^0$ has zero object $A\stackrel{1_A}\to A$.

We claim that the inclusion $(A\downarrow\F)_c^0\hookrightarrow (A\downarrow\F)_c$ induces an equivalence of quotient categories $(A\downarrow\F)_c^0/\sim \simeq (A\downarrow\F)_c/\sim$. Firstly, note that $A\stackrel{\left(\begin{smallmatrix}
1 \\
0\\
0
\end{smallmatrix}\right)}\to A\oplus C\oplus I(C)$ is a very good cylinder object of $u=A\stackrel{\left(\begin{smallmatrix}
1 \\
0
\end{smallmatrix}\right)}\to A\oplus C$:

\[
\xymatrix{& A\ar@/_12pt/[dl]_{\left(\begin{smallmatrix}
1 \\
0 \\
0
\end{smallmatrix}\right)} \ar[d]^{\left(\begin{smallmatrix}
1 \\
0 \\
0
\end{smallmatrix}\right)} \ar@/^12pt/[dr]^{\left(\begin{smallmatrix}
1 \\
0 
\end{smallmatrix}\right)}&\\
A\oplus C\oplus C \ar[r]^{\left(\begin{smallmatrix}
1 & 0 & 0\\
0 & 1 & 1\\
0 & i & 0
\end{smallmatrix}\right)\ \ }& A\oplus C\oplus I(C) \ar[r]^{\ \ \ \ \ \left(\begin{smallmatrix}
1 & 0 & 0\\
0 & 1 & 0
\end{smallmatrix}\right)}& A\oplus C}
\] 
where $C\stackrel{i}\to I(C)$ is an injective preenvelope of $C$ in $\F$. Given any morphism from $\left(\begin{smallmatrix}
1 & r\\
0 & s
\end{smallmatrix}\right):u=A\stackrel{\left(\begin{smallmatrix}
1 \\
0
\end{smallmatrix}\right)}\to A\oplus C\to v=A\stackrel{\left(\begin{smallmatrix}
1 \\
0
\end{smallmatrix}\right)} \to A\oplus D$, there is a morphism $r':I(C)\to A$ such that $r'i=r$ since $A$ is injective. Then $\left(\begin{smallmatrix}
1 & 0 & r'\\
0 & s & 0
\end{smallmatrix}\right)$ is a cylinder homotopy from $\left(\begin{smallmatrix}
1 & r\\
0 & s
\end{smallmatrix}\right)$ to $\left(\begin{smallmatrix}
1 & 0\\
0 & s
\end{smallmatrix}\right)$ in $(A\downarrow \F)_c$. That is $\left(\begin{smallmatrix}
1 & r\\
0 & s
\end{smallmatrix}\right)\sim \left(\begin{smallmatrix}
1 & 0\\
0 & s
\end{smallmatrix}\right)$ in $(A\downarrow \F)_c$. Meanwhile it is easy to prove that given any two morphisms $\left(\begin{smallmatrix}
1 & \\
0 & s
\end{smallmatrix}\right), \left(\begin{smallmatrix}
1 & 0\\
0 & t
\end{smallmatrix}\right):u\to v$ in $(A\downarrow\F)_c^0$, they are cylinder homotopic if and only if they cylinder homotopic in $(A\downarrow \F)_c$. So we have $(A\downarrow\F)_c/\sim\simeq (A\downarrow\F)_c^0/\sim$. In particular $\H o(A\downarrow\F)\simeq\H o((A\downarrow\F)_c^0)$.

Now we can use Quillen's Theorem I.2 of \cite{Quillen67} to the homotopy category $\H o((A\downarrow \F)_c^0)$. Recall that $A\stackrel{\left(\begin{smallmatrix}
1 \\
0\\
0
\end{smallmatrix}\right)}\to A\oplus C\oplus I(C)$ is a very good cylinder object of $u=A\stackrel{\left(\begin{smallmatrix}
1 \\
0
\end{smallmatrix}\right)}\to A\oplus C$, where $I(C)$ is an injective preenvelope of $C$. By the construction of the suspension functor of the homotopy category $\H o(A\downarrow \F)$ \cite[The proof of Theorem 2 in Chapter I]{Quillen67}, we may define $\Sigma(u)=A\stackrel{\left(\begin{smallmatrix}
1 \\
0
\end{smallmatrix}\right)}\to A\oplus \Sigma^\F(C)$. Where $\Sigma^\F$ is the suspension functor of the stable category $\underline{\F}$ which is an automorphism of $\underline{\F}$; see \cite[Chapter I, Proposition 2.2]{Happel88}.  Then it can be verified directly that the suspension functor $\Sigma$ defined as above on the homotopy category $\H o((A\downarrow \F)_c^0)$ is an auto-equivalence and thus $\H o(A\downarrow \F)$ is a triangulated categories by Proposition 5-6 in Section I.3 of \cite{Quillen67}.\end{proof}

\begin{corollary}  The derived functor $\mathbb{L}0_!: \underline{\F}\to \H o(A\downarrow \F)$ is a triangle equivalence with quasi-inverse $\mathbb{R}0^*$.
\end{corollary}
\begin{proof} \ By construction, $(0_!, 0^*):\F\to (A\downarrow \F)$ is a Quillen equivalence. Then the derived adjunction $(\mathbb{L}0_!,\mathbb{R}0^*)$ is an equivalent adjunction and $\mathbb{L}0_!$ is a triangle equivalence by Theorem I.3 of \cite{Quillen67}.\end{proof}

\begin{remark} \ Dually, we can construct a slice category $(\F\downarrow A)$ for a nonzero projective-injective object $A$, and there is a triangle equivalence $(\mathbb{L}0_*,\mathbb{R}0^!): \underline{\F}\to \H o(\F\downarrow A)$.
\end{remark}

%\begin{corollary} \ Let $\A$ be a WIC exact model category with three injective model structures $\M_1, \ \M_2, \ \M_3$ such that $\M_{2f}\cup\M_{3f}\subseteq \M_{1f}$. Assume that $\M_{1f}\cap \M_{3,triv}=\M_{2f}$. Then $1$ \ for any object $A\in \A$, we have Quillen equivalences  $$(X\downarrow \M_1/\M_2)\underset{{\rm Id}}{\stackrel{{\rm Id}}\rightleftarrows} (X\downarrow\M_3) \ \mbox{and} \ (X\downarrow \M_1/\M_3)\underset{{\rm Id}}{\stackrel{{\rm Id}}\rightleftarrows}(X\downarrow \M_2).$$
%\vskip5pt
%$(2)$ \  we have a recollement:
%\end{corollary}

 \noindent{\bf Acknowledgements} \ The author would like to thank Xiao-Wu Chen, Guodong Zhou and Xiaojin Zhang for their helpful discussions and encouragements.

\vskip10pt
{\footnotesize \noindent Zhi-Wei Li \\
 School of Mathematics and Statistics, \\
 Jiangsu Normal University,
Xuzhou 221116, Jiangsu, PR China.\\
{\it E-mail address: zhiweili$\symbol{64}$jsnu.edu.cn}.}

\end{document}